\documentclass[english]{article}
\usepackage[T1]{fontenc}
\usepackage[latin9]{inputenc}
\usepackage{geometry}
\usepackage{amsthm}
\usepackage{amsmath}
\usepackage{amssymb}
\usepackage{graphicx}
\usepackage{esint}

\makeatletter
\theoremstyle{plain}
\newtheorem{thm}{\protect\theoremname}[section]
  \theoremstyle{plain}
  \newtheorem{conjecture}[thm]{\protect\conjecturename}
  \theoremstyle{plain}
  \newtheorem{prop}[thm]{\protect\propositionname}
  \theoremstyle{remark}
  \newtheorem*{acknowledgement*}{\protect\acknowledgementname}
  \theoremstyle{plain}
  \newtheorem{lem}[thm]{\protect\lemmaname}

\newcommand{\hide}[1]{}

\DeclareMathOperator{\diam}{diam}

\DeclareMathOperator{\vol}{vol}

\DeclareMathOperator{\sdim}{sdim}

\DeclareMathOperator{\bdim}{dim_\textrm{B}}

\newtheorem{almost-theorem}[thm]{Almost-theorem}

\usepackage{wrapfig}
\usepackage{graphicx}

\theoremstyle{plain}
\newtheorem{fprop}[thm]{``Proposition''}

\makeatletter
\def\blfootnote{\xdef\@thefnmark{}\@footnotetext}
\makeatother

\newcommand{\rev}[1]{#1}

\makeatother

\usepackage{babel}
  \providecommand{\acknowledgementname}{Acknowledgement}
  \providecommand{\conjecturename}{Conjecture}
  \providecommand{\lemmaname}{Lemma}
  \providecommand{\propositionname}{Proposition}
\providecommand{\theoremname}{Theorem}

\begin{document}

\title{Self similar sets, entropy and additive combinatorics }

\author{Michael Hochman}
\maketitle
\begin{abstract}
This\blfootnote{Supported by ERC grant 306494}\blfootnote{\emph{2010 Mathematics Subject Classication}. 28A80, 11K55, 11B30, 11P70}
article is an exposition of the main result of \cite{Hochman2012a},
that self-similar sets whose dimension is smaller than the trivial
upper bound have ``almost overlaps'' between cylinders. We give
a heuristic derivation of the theorem using elementary arguments about
covering numbers. We also give a short introduction to additive combinatorics,
focusing on inverse theorems, which play a pivotal role in the proof.
Our elementary approach avoids many of the technicalities in \cite{Hochman2012a},
but also falls short of a complete proof; in the last section we discuss
how the heuristic argument is turned into a rigorous one.
\end{abstract}

\section{\label{sec:Intro}Introduction}

\subsection{\label{sub:Self-similar-sets}Self-similar sets}

Self-similar sets in the line are compact sets that are composed of
finitely many scaled copies of themselves. These are the simplest
fractal sets, the prototypical example being the famous middle-$\frac{1}{3}$
Cantor set $X\subseteq[0,1]$, which satisfies the ``geometric recursion''%
\footnote{The mddle-1/3 Cantor set can also be described in other ways, e.g.
by a recursive construction, or symbolically as the points in $[0,1]$
that can be written in base $3$ without the digit $1$. General self-similar
sets also have representations of this kind, but in this paper we
shall not use them.%
} relation $X=\frac{1}{3}X\cup(\frac{1}{3}X+\frac{2}{3})$, using the
obvious notation for scaling and translation of a set. In general,
a self-similar set \rev{in $\mathbb{R}$} is defined by a finite family
$\Phi=\{f_{i}\}_{i\in\Lambda}$ of maps of the form $f_{i}(x)=r_{i}x+a_{i}$,
where $0<|r_{i}|<1$ and $a_{i}\in\mathbb{R}$. The family $\Phi$
is called an \emph{iterated function system }(or\emph{ IFS}),%
\footnote{Iterated function systems consisting of non-affine maps and on other
metric spaces than $\mathbb{R}$ are also of interest, but we do not
discuss them here.%
}\emph{ }and the self-similar set they define is unique compact set
$X\neq\emptyset$ satisfying 
\begin{equation}
X=\bigcup_{i\in\Lambda}f_{i}X.\label{eq:SSSs}
\end{equation}
(existence and uniqueness are due to Hutchinson \cite{Hutchinson1981}). 

Throughout this paper we make a few simplifying assumptions. To avoid
trivialities, we always assume that $\Phi$ contains at least two
distinct maps, otherwise $X$ is just the common fixed point of the
maps. We assume that $\Phi$ has \emph{uniform contraction}, i.e.
all the contraction ratios $r_{i}$ are equal to the same value $r$.
Finally, we assume that $r>0$, so the maps preserve orientation.
These assumptions are not necessary but they simplify the statements
and arguments considerably.

\subsection{\label{sub:Conjecture}Dimension of self-similar sets}

Despite the apparent simplicity of the definition, and of some of
the better known examples, there are still large gaps in our understanding
of the geometry of self-similar sets. In general, we do not even know
how to compute their dimension. Usually one should be careful to specify
the notion of dimension that one means, but it is a classical fact
that, for self-similar sets, all the major notions of dimension coincide,
and in particular the Hausdorff and box (Minkowski) dimensions agree
(e.g. \cite[Theorem 4 and Example 2]{Falconer1989}). Thus we are
free to choose either one of these, and we shall choose the latter,
whose definition we now recall. For a subset $Y\subseteq\mathbb{R}$
denote its covering number at scale $\varepsilon$ by
\[
N_{\varepsilon}(Y)=\min\{k\,:\, Y\mbox{ can be covered by }k\mbox{ sets of diameter }\leq\varepsilon\}
\]
The \emph{box dimension }of $Y$, if it exists, is the exponential
growth rate of $N_{\varepsilon}(Y)$:
\[
\bdim Y=\lim_{\varepsilon\rightarrow0}\frac{\log N_{\varepsilon}(Y)}{\log(1/\varepsilon)}
\]
Thus $\bdim Y=\alpha$ means that $N_{\varepsilon}(Y)=\varepsilon^{-\alpha+o(1)}$
as $\varepsilon\rightarrow0$. It is again well known that the limit
exists when $Y$ is self-similar, we shall see a short proof in Section
\ref{sub:Sumset-structure-of-SSSs}.

It is easy to give upper bounds for the dimension of a self-similar
set. Taking $X$ as in (\ref{eq:SSSs}) and iterating the relation
we get 
\[
X=\bigcup_{i\in\Lambda}f_{i}(\bigcup_{j\in\Lambda}f_{j}X)=\bigcup_{i,j\in\Lambda}f_{i}\circ f_{j}(X)
\]
Writing $f_{i_{1}\ldots i_{n}}=f_{i_{1}}\circ\ldots\circ f_{i_{n}}$for
$\underline{i}=i_{1}\ldots i_{n}\in\Lambda^{n}$ and iterating  $n$
times, we have 
\begin{equation}
X=\bigcup_{\underline{i}\in\Lambda^{n}}f_{\underline{i}}X\label{eq:iterated-recursion}
\end{equation}
This union consists of $|\Lambda|^{n}$ sets of diameter $r^{n}|X|$,
so by definition,
\[
N_{r^{n}\diam(X)}(X)\leq|\Lambda|^{n}
\]
Hence
\begin{equation}
\bdim(X)=\lim_{n\rightarrow\infty}\frac{\log N_{r^{n}\diam(X)}(X)}{\log(1/r^{n}\diam(X))}\leq\frac{\log|\Lambda|}{\log(1/r)}\label{eq:trivial-bound}
\end{equation}
The right hand side of (\ref{eq:improved-trivial-bound}) is called
the \emph{similarity dimension} of $X$ and is denoted $\sdim X$.%
\footnote{It would be better to write $\sdim\Phi$, since this quantity depends
on the presentation of $X$ and not on $X$ itself, but generally
there is only one IFS given and no confusion should arise.%
}

Is this upper bound an equality? Note that the bound is purely combinatorial
and does not take into account the parameters $a_{i}$ at all. Equality
is known to hold under some assumptions on the separation of the ``pieces''
$f_{i}X$, $i\in\Lambda$, for instance assuming \emph{strong separation}
(that the union (\ref{eq:SSSs}) is disjoint), or the open set condition
(that there exists open set $\emptyset\neq U\subseteq\mathbb{R}$
such that $f_{i}U\subseteq U$ and $f_{i}U\cap f_{j}U=\emptyset$
for $i\neq j$). 

Without separation conditions, however, the inequality (\ref{eq:trivial-bound})
can be strict. There are two trivial ways this can occur. First, there
could be too many maps: if $|\Lambda|>1/r$ then the right hand side
of (\ref{eq:trivial-bound}) is greater then $1$, whereas $\bdim X\leq1$
due to the trivial bound $N(X,\varepsilon)\leq\left\lceil \diam(X)/\varepsilon\right\rceil $.
Thus we should adjust (\ref{eq:trivial-bound}) to read
\begin{equation}
\bdim(X)\leq\min\{1,\frac{\log|\Lambda|}{\log(1/r)}\}\label{eq:improved-trivial-bound}
\end{equation}

Second, the combinatorial bound may be over-counting if some of the
sets in the union (\ref{eq:iterated-recursion}) coincide, that is,
for some $n$ we have $f_{\underline{i}}=f_{\underline{j}}$ for some
distinct $\underline{i},\underline{j}\in\Lambda^{n}$. This situation
is known as \emph{exact overlaps}. If such $\underline{i},\underline{j}$
exist then we can re-write (\ref{eq:iterated-recursion}) as $X=\bigcup_{\underline{u}\in\Lambda^{n}\setminus\{\underline{i}\}}f_{\underline{u}}X$,
which presents $X$ as the attractor of the IFS $\Phi'=\{f_{\underline{u}}\}_{\underline{u}\in\Lambda'}$
for $\Lambda'=\Lambda^{n}\setminus\{\underline{i}\}$. This IFS consists
of $|\Lambda'|=|\Lambda|^{n}-1$ maps that contract by $r^{n}$, so,
applying the trivial bound (\ref{eq:trivial-bound}) to this IFS,
we have $\bdim X\leq\log(|\Lambda|^{n}-1)/\log(1/r^{n})$, which is
better than the previous bound of $\log|\Lambda|/\log(1/r)$. To take
an extreme example, if all the maps $f_{i}$ coincide then the attractor
$X$ is just the unique fixed point of the map, and its dimension
is $0$. 

Are there other situations where a strict inequality occurs in (\ref{eq:improved-trivial-bound})?
A-priori, one does not need \emph{exact} coincidences between sets
in (\ref{eq:iterated-recursion}) to make the combinatorial bound
very inefficient. It could happen, for example, that many of the sets
$f_{\underline{i}}X$, $\underline{i}\in\Lambda^{n}$, align almost
exactly, in which case one may need significantly fewer than $|\Lambda|^{n}$
$\varepsilon$-intervals to cover them. Nevertheless, although such
a situation can easily be arranged for a fixed $n$, to get a drop
in dimension one would need this to happen at all sufficiently small
scales. No such examples are known, and the main subject of this paper
is the conjecture that this cannot happen:
\begin{conjecture}
\label{conj:main}A strict inequality in (\ref{eq:improved-trivial-bound})
can occur only in the presence of exact overlaps.
\end{conjecture}
This conjecture appears in \cite[Question 2.6]{PeresSolomyak2000b},
though special cases of it have received attention for decades, in
particular Furstenberg's projection problem for the 1-dimensional
Sierpinski gasket (see e.g. \cite{Kenyon97}), the 0,1,3-problem (see
e.g. \cite{PollicottSimon1995}) and, for self-similar measures instead
of sets, the Bernoulli convolutions problem, (e.g. \cite{PeresSchlagSolomyak00}).

One may also draw an analogy between this conjecture and rigidity
statements in ergodic theory. Rigidity is the phenomenon that, for
certain group actions of algebraic origin, the orbit of the point
is as large as it can be (dense or possibly even equidistributed for
the volume measure) unless there is an algebraic obstruction to this
happening. To see the connection with the conjecture above, note that
$X$ is just the orbit closure of (any) $x\in\mathbb{R}$ under the
semigroup $\{f_{\underline{i}}\,:\,\underline{i}\in\bigcup_{n=1}^{\infty}\Lambda^{N}\}$
of affine maps, and that exact overlaps occur if and only if this
semigroup is not generated freely by $\{f_{i}\}_{i\in\Lambda}$. Thus
the conjecture predicts that the orbit closure of any point is as
large as it can be unless there are algebraic obstructions.

\subsection{\label{sub:Main-result}Progress towards the conjecture}

Our main subject here is a weakened form of the conjecture which proves
the full conjecture in some important examples and special classes
of IFSs. In order to state it we must first quantify the degree to
which the sets $f_{\underline{i}}(X)$ are separated from each other.
Since all of the maps in $\Phi$ contract by the same ratio, any two
of the sets $f_{\underline{i}}(X)$ and $f_{\underline{j}}(X)$ for
$\underline{i}.\underline{j}\in\Lambda^{n}$ are translates of each
other. We define the distance between them as the magnitude of this
translation, which is given by $f_{\underline{i}}(x)-f_{\underline{j}}(x)$
for any $x\in\mathbb{R}$; we shall choose $x=0$ for concreteness.
Thus a measure of the degree of concentration of cylinders $f_{\underline{i}}(X)$,
$\underline{i}\in\Lambda^{n}$, is provided by 
\begin{eqnarray*}
\Delta_{n} & = & \min\{|f_{\underline{i}}(0)-f_{\underline{j}}(0)|\,:\,\underline{i},\underline{j}\in\Lambda^{n}\,,\,\underline{i}\neq\underline{j}\}
\end{eqnarray*}
Evidently, exact overlaps occur if and only if there exists an $n$
such that $\Delta_{n}=0$. Fixing $x\in X$, the points $f_{\underline{i}}(x)$
, $\underline{i}\in\Lambda^{n}$, all lie in $X$, and so there must
be a distinct pair $\underline{i},\underline{j}\in\Lambda^{n}$ with
$|f_{\underline{i}}(x)-f_{\underline{j}}(x)|\leq\diam(X)/|\Lambda|^{n}$;
hence $\Delta_{n}\rightarrow0$ at least exponentially. In general
there may be an exponential lower bound on $\Delta_{n}$ as well,
i.e. $\Delta_{n}\geq cr^{n}$ for some $c,r>0$. This is always the
case when the IFS satisfies strong separation or the open set condition,
but there are examples where it holds even when these conditions fail
(see Garsia \cite{Garsia1962}). Therefore the following theorem from
\cite{Hochman2012a} gives nontrivial information and should be understood
as a weak form of Conjecture \ref{conj:main}.
\begin{thm}
\label{thm:main}If $X\subseteq\mathbb{R}$ is a self-similar set
and $\dim X<\min\{1,\sdim X\}$, then $\Delta_{n}\rightarrow0$ super-exponentially,
that is, $-\frac{1}{n}\log\Delta_{n}\rightarrow\infty$. 
\end{thm}
In practice, one applies the theorem after establishing an exponential
lower bound on $\Delta_{n}$ to deduce that $\dim X=\min\{1,\sdim X\}$.
For example,
\begin{prop}
Let $\mathcal{R}$ denote the set of rational IFSs, i.e. such that
$r,a_{i}\in\mathbb{Q}$. Then Conjecture \ref{conj:main} holds in
$\mathcal{R}$.\end{prop}
\begin{proof}
First,~a useful identity: For $\underline{i}\in\Lambda^{n}$, a direct
calculation shows that 
\begin{eqnarray}
f_{\underline{i}}(x) & = & r^{n}x+\sum_{k=1}^{n}a_{i_{k}}r^{k-1}\label{eq:iterated-f-a}\\
 & = & r^{n}x+f_{\underline{i}}(0)\label{eq:iterated-f-b}
\end{eqnarray}
Now let that $f_{i}(x)=rx+a_{i}$ where $r=p/q$ and $a_{i}=p_{i}/q_{i}$
for $p,p_{i},q,q_{i}$ integers and write $Q=\prod_{i\in\Lambda}q_{i}$.
Then $f_{\underline{i}}(0)=\sum_{k=1}^{n}a_{i_{k}}r^{n-k}$ is a rational
number with denominator $Qq^{n}$. Suppose that no overlaps occur,
so that $\Delta_{n}>0$ for all $n$. Given $n$, by definition there
exist distinct $\underline{i},\underline{j}\in\Lambda^{n}$ such that
$\Delta_{n}=f_{\underline{i}}(0)-f_{\underline{j}}(0)$. Therefore
$\Delta_{n}$ is a non-zero rational number with denominator $Qq^{n}$
so we must have $\Delta_{n}\geq1/Qq^{n}$. By Theorem \ref{thm:main}
we conclude that $\dim X=\min\{1,\sdim X\}$.
\end{proof}
The same argument works in the class of IFSs with algebraic coefficients,
using a similar lower bound on polynomial expressions in a given set
of algebraic numbers. See \cite[Theorem 1.5]{Hochman2012a}. A simple
(but non-trivial) calculation, due to B. Solomyak and P. Shmerkin,
also allows one to deal with the case that one of the translation
parameters $a_{i}$ is irrational, resolving Furstenberg's question
about linear projections of the one-dimensional Sierpinski gasket
\cite[Theorem 1.6]{Hochman2012a}. Theorem \ref{thm:main} leads to
strong results about parametric families of self-similar sets \cite[Theorem 1.8]{Hochman2012a},
and there is a version for measures which has also led to substantial
progress on the Bernoulli convolutions problem, see \cite[Theorem 1.9]{Hochman2012a}
and the recent advance by Shmerkin \cite{Shmerkin2013}. Another interesting
application is given in \cite{Orponen2013}. 

The rest of this paper is an exposition of the proof of the theorem.
Our goal is to present the ideas as transparently as possible, and
to this end we frame the argument in terms of covering numbers (rather
than entropy as in \cite{Hochman2012a}). This leads to simpler statements
and to an argument that is conceptually correct but, unfortunately,
incomplete; some crucial steps of this simplified argument are flawed.
In spite of this deficiency we believe that such an exposition will
be useful as a guide to the more technical proof in \cite{Hochman2012a}.
To avoid any possible misunderstandings, we have indicated the false
statements in quotation marks (``Lemma'', ``Proof'', etc.).

As we shall see, the main idea is to reduce (the negation of) the
theorem to a statement about sums of self-similar sets with other
sets. Problems about sums of sets fall under the general title of
additive combinatorics, and in the next section we give a brief introduction
to the parts of this theory that are relevant to us. In Section \ref{sec:Main-reduction-heuristic}
we explain the reduction to a statement about sumsets, and show how
an appropriate inverse theorem essentially settles the matter. Finally,
in Section \ref{sec:From-Proof-to-Proof}, we discuss how the heuristic
argument can be made rigorous. 
\begin{acknowledgement*}
Many thanks to Boris Solomyak for his coments on the paper.
\end{acknowledgement*}

\section{\label{sec:Additive-combinatorics}A birds-eye view of additive combinatorics}

\subsection{\label{sub:Freiman}Sumsets and inverse theorems}

The sum (or sumset) of non-empty sets $A,B\subseteq\mathbb{R}^{d}$
is 
\[
A+B=\{a+b\,:\, a\in A\,,\, b\in B\}
\]
Additive combinatorics, or at least an important chapter of it, is
devoted to the study of sumsets and the relation between the structure
of $A,B$ and $A+B$. We focus here on so-called inverse problems,
that is the problem of describing the structure of sets $A,B$ such
that $A+B$ is ``small'' relative to the sizes of the original sets.
The general flavor of results of this kind is that, if the sumset
is small, there must be an algebraic or geometric reason for it. It
will become evident in later sections that this question comes up
naturally in the study of self-similar sets.

To better interpret what ``small'' means, first consider the trivial
bounds. Assume that $A,B$ are finite and non-empty. Then $|A+B|\geq\max\{|A|,|B|\}$,
with equality if and only if at least one of the sets is a singleton.
In the other direction, $|A+B|\leq|A||B|$, and equality can occur
(consider $A=\{0,10,20,30,\ldots,10n\}$ and $B=\{0,1,\ldots,9\}$).
For ``generic'' pairs of sets the upper bound is close to the truth.
For example, when $A,B\subseteq\{1,\ldots,n\}$ are chosen randomly
by including each $1\leq i\leq n$ in $A$ with probability $p$ and
similarly for $B$, with all choices independent, there is high probability
that $|A+B|\geq c|A||B|$. The question becomes, what can be said
between these two extremes.

\subsection{\label{sub:Minimal-growth}Minimal growth}

One of the earliest inverse theorems is the Brunn-Minkowski inequality
of the late 19th century. The setting is $\mathbb{R}^{d}$ with the
volume measure, and it states that if $A,B\subseteq\mathbb{R}^{d}$
are convex sets then, given the volumes of $A,B$, the volume of $A+B$
is minimized when $A,B$ are balls with respect to some common norm.
Since the volume of a ball scales like the $d$-th power of the radius,
this means that $\vol(A+B)\geq(\vol(A)^{1/d}+\vol(B)^{^{1/d}})^{d}$,
and equality occurs if and only if, up to a nullset, $A,B$ are dilates
of the same convex set. The inequality was later extended to arbitrary
Borel sets (note that $A+B$ may not be a Borel set but it is an analytic
set and hence Lebesgue measurable). For a survey of this topic see
Gardner \cite{Gardner2002}.

Similar tight statements hold in the discrete setting. The analog
of a convex body is an \emph{arithmetic progression (AP)}, namely
a set of the form $P=\{a,a+d,a+2d,\ldots,a+(1-k)d\}$, where $d$
is called the gap (we assume $d\neq0$) and $k$ is called the length
of $P$. Then for finite sets $A,B\subseteq\mathbb{Z}^{d}$ with $|A|,|B|\geq2$
we always have $|A+B|\geq|A|+|B|-1$, with equality if and only if
$A,B$ are APs of the same gap \cite[Proposition 5.8]{TaoVu2006}.

\subsection{\label{sub:Linear-growth:-small-doubling}Linear growth: small doubling
and Freiman's theorem}

Now suppose that $A=B\subseteq\mathbb{Z}^{d}$ but weaken the hypothesis,
assuming only that 
\begin{equation}
|A+A|\leq C|A|\label{eq:freiman-growth}
\end{equation}
where we think of $A$ as large and $C$ as constant. Such sets are
said to have \emph{small doubling}. 

The simplest example of small doubling in $\mathbb{Z}^{d}$ is when
$A=\{1,\ldots,n\}^{d},$ in which case $|A+A|\leq2^{d}|A|$. This
example can be pushed down to any lower dimension as follows. For
$i=1,\ldots,k$, take intervals of integers $I_{i}=\{1,2,\ldots,n_{i}\}$,
and let $T:\mathbb{Z}^{k}\rightarrow\mathbb{Z}^{d}$ be an affine
map given by integer parameters. Suppose that $T$ is injective on
$I=I_{1}\times\ldots\times I_{k}$. Then $A=T(I)\subseteq\mathbb{Z}^{d}$
has the property that 
\[
|A+A|=|T(I)+T(I)|=|T(I+I)|\leq|I+I|\leq2^{k}|I|=2^{k}|A|
\]
A set $A$ as above is called a \emph{(proper) generalized arithmetic
progression} \emph{(GAP) of rank $k$}. 

GAPs are still extremely algebraic objects but one can get away from
this a little using another cheap trick: Begin with a set $A$ satisfying
$|A+A|\leq C|A|$ (e.g. a GAP) and choose any $A'\subseteq A$ with
cardinality $|A'|\geq D^{-1}|A|$ for some $D>1$. Then 
\[
|A'+A'|\leq|A+A|\leq C|A|\leq CD|A'|
\]

One of the central results of additive combinatorics is Freiman's
theorem, which says that, remarkably, these are the only ways to get
small doubling.
\begin{thm}
[Freiman] \label{thm:Freiman} If $A\subseteq\mathbb{Z}^{d}$ and
$|A+A|\leq C|A|$, then $A\subseteq P$ for a GAP $P$ of rank $C'$
and satisfying $|P|\leq C''|A|$. The constants satisfy  $C'=O(C(1+\log C))$
and $C''=C^{O(1)}$.
\end{thm}
For more information see \cite[Theorem 5.32 and Theorem 5.33]{TaoVu2006}. 

Combined with some standard arguments (e.g. the Pl\"{u}nnecke-Rusza
inequality), the symmetric version leads to an asymmetric versions:
assuming $A,B\subseteq\mathbb{Z}^{d}$ and $C^{-1}\leq|A|/|B|\leq C$,
if $|A+B|\leq C|A|$ then $A,B$ are contained in a GAP $P$ of rank
and $\leq C'$ and size $|P|\leq C'|A|$, with similar bounds on the
constants.

\subsection{\label{sub:Sumsets-with-power-growth}Power growth, the ``fractal''
regime}

Now relax the growth condition even more and consider finite sets
$A\subseteq\mathbb{Z}$ (or $A\subseteq\mathbb{R}$) such that 
\begin{equation}
|A+A|\leq|A|^{1+\delta}\label{eq:power-bound}
\end{equation}
This is the discrete analog of the condition 
\begin{equation}
\bdim(X+X)\leq(1+\delta)\bdim X\label{eq:dim-growth}
\end{equation}
for $X\subseteq\mathbb{R}$. Indeed, given $X\subseteq\mathbb{R}$
and $n\in\mathbb{N}$ let $X_{n}$ denote the set obtained by replacing
each $x\in X$ with the closest point $k/2^{n}$, $k\in\mathbb{Z}$.
Then $|X_{n}|\sim2^{n(\bdim X+o(1))}$ and $|X_{n}+X_{n}|\sim2^{n(\bdim(X+X)+o(1))}$
for large $n$, so (\ref{eq:dim-growth}) is equivalent to $|X_{n}+X_{n}|\lesssim|X_{n}|^{1+o(1)}$.
Thus, the difference between (\ref{eq:freiman-growth}) and (\ref{eq:power-bound})
is roughly the difference between using Lebesgue measure or dimension
to quantify the size of a set $X\subseteq\mathbb{R}$.

Here is a typical example of a set satisfying (\ref{eq:power-bound}).
Write $P_{n}=\{0,\ldots,n-1\}$ and let 
\begin{eqnarray*}
A_{n} & = & \sum_{i=1}^{n}\frac{1}{2^{i^{2}}}P_{2^{i}}\\
 & = & \{\sum_{i=1}^{n}a_{i}2^{-i^{2}}\,:\,1\leq a_{i}\leq2^{i}\}
\end{eqnarray*}
(again, one can think of this either as a subset of $\mathbb{R}$,
or of $\frac{1}{4^{n^{2}}}\mathbb{Z}$). It is easy to verify that
the distance between distinct points $x,x'\in A_{n}$ is at least
$1/4^{n^{2}}$, and that such $x$ has a unique representation as
a sum $\sum_{i=1}^{n}a_{i}4^{-i^{2}}\,:\,1\leq a_{i}\leq2^{i}$. Indeed,
each term in the sum $\sum_{i=1}^{n}a_{i}2^{-i^{2}}$ determines a
distinct block of binary digits.  Thus $A_{n}$ is a GAP, being the
image of $P_{2}\times P_{4}\times\ldots\times P_{2^{n}}$ by the map
$(x_{1},\ldots,x_{n})\mapsto\sum\frac{1}{2^{i^{2}}}x_{i}$. The rank
is $n$, and so, as we saw in the previous section, 
\[
|A_{n}+A_{n}|\leq2^{n}|A_{n}|
\]
Since
\[
|A_{n}|=\prod_{i=1}^{n}|P_{n}|=2^{\sum_{i=1}^{n}i}=2^{n(n+1)/2}
\]
we have
\[
|A_{n}+A_{n}|=|A_{n}|^{1+o(1)}\qquad\mbox{as }n\rightarrow\infty
\]

The reader may recognize the example above as the discrete analog
of a Cantor set construction, where at stage $n$ we have a collection
of intervals $2^{n(n+1)/2}$ of length $2^{-n^{2}}$, and from each
of these intervals we keep $2^{n+1}$ sub-intervals of length $2^{-(n+1)^{2}}$,
separated by gaps of length $2^{-n^{2}-(n+1)}$. For the resulting
Cantor set $X$ it is a standard exercise to see that $\dim X=\bdim X=1/2$,
and the calculation above shows that $\dim X+X=1/2$ as well. Such
constructions appear in the work of Erd\H{o}s-Volkmann \cite{ErdosVolkmann1966},
and also in the papers of Schmeling-Shmerkin \cite{SchmelingShmerkin2009}
and K\"{o}rner \cite{Korner2008}, who showed that for any sequence
$\alpha_{1}\leq\alpha_{2}\leq\ldots$ there is a set $X$ with $\dim\sum_{i=1}^{n}X=\alpha_{n}$.

Do all examples of (\ref{eq:dim-growth}) look essentially like this
one? In principle one can apply Freiman's theorem, since the hypothesis
(\ref{eq:dim-growth}) can be written as $|A+A|\leq C|A|$ for $C=|A|^{\delta}$.
What one gets, however, is that $A$ is a $|A|^{O(\delta)}$-fraction
of a GAP or rank $|A|^{O(\delta)}$, and this gives rather coarse
information about $A$ (note that, trivially, every set is a GAP of
rank $|A|$).

Instead, it is possible to apply a multi-scale analysis, showing that
at some scales the set looks quite ``dense'' and at others quite
``sparse''. The best way to explain this is in the language of trees,
which we introduce next.

\subsection{\label{sub:Trees-and-tree-measures}Trees and tree-measures }

Denote the length of a finite sequence $\sigma=\sigma_{1}\ldots\sigma_{n}$
by $|\sigma|=n$ and write $\emptyset$ for the empty word, which
by definition has $|\emptyset|=0$. Denote the concatenation of words
$\sigma$ and $\tau$ by $\sigma\tau$, in which case we say that
$\sigma$ is a prefix of $\tau$, and that $\sigma\tau$ extends $\sigma$. 

The \emph{full binary tree }of height $h$ is the set $\{0,1\}^{\leq h}=\bigcup_{k=0}^{h}\{0,1\}^{k}$
of $0,1$-valued sequences of length $\leq h$, where our convention
is that $\{0,1\}^{0}=\{\emptyset\}$, so the empty word is included.
We define a \emph{tree of height }$h$ is a subset $T\subseteq\bigcup_{i=0}^{h}\{0,1\}^{i}$
satisfying 
\begin{enumerate}
\item [(T1)] $\emptyset\in T$.
\item [(T2)] If $\sigma\in T$ and $\eta$ is an initial segment of $\sigma$
then $\eta\in T$.
\item [(T3)] If $\sigma\in T$ then there is an $\eta\in T$ which extends
\ $\sigma$ and $|\eta|=h$.
\end{enumerate}
One may think of $T$ as a set of vertices and introduce edges between
every pair $\sigma_{1}\ldots\sigma_{i},\sigma_{1}\ldots\sigma_{i}\sigma_{i+1}\in T$.
Then $T$ is a tree if $\emptyset\in T$ and in the associated graph
there is a path from $\emptyset$ to every node, and all maximal paths
are of length $h$.

The \emph{level (or depth)} of $\sigma\in T$ is its length (the graph-distance
from $\emptyset$ to $\sigma$). The \emph{leaves }of a tree $T$
of height $h$ are the elements of the lowest (deepest) level, namely
$h$: 
\[
\partial T=T\cap\{0,1\}^{h}
\]
The \emph{descendants }of $\sigma\in T$ are the nodes $\eta\in T$
that extend $\sigma$. The nodes $m$ generations below $\sigma$
in $T$ are the nodes of the form $\eta=\sigma\sigma'\in T$ for $\sigma'\in\{0,1\}^{m}$.

We also shall need to work with measures ``on trees'', or, rather,
measures on their leaves. For notational purposes it is better to
introduce the notion of a \emph{tree-measure}%
\footnote{This notion is identical to a flow on the tree in the sense of network
theory.%
} on the full tree $\{0,1\}^{\leq h}$, namely, a function $\mu:\{0,1\}^{\leq h}\rightarrow[0,1]$
satisfying 
\begin{enumerate}
\item [(M1)] $\mu(\emptyset)=1$.
\item [(M2)] $\mu(\sigma)=\sum_{i\in\{0,1\}}\mu(\sigma i)$
\end{enumerate}
It is easily to derive from (M1) and (M2) that $\sum_{\sigma\in\{0,1\}^{k}}\mu(\sigma)=1$
for every $1\leq k\leq h$, so a tree-measure induces genuine probability
measures on every level of the full tree, and in particular on $\partial T$.
Conversely, if we have a genuine probability measure $\mu$ on the
set of leaves $\{0,1\}^{h}$ of the full tree of height $h$ then
it induces a tree-measure by $\mu(\sigma)=\sum_{\eta\,:\,\sigma\eta\in\partial T}\mu(\{\sigma\eta\})$.
Given a tree-measure, the set $T=\{\sigma\,:\,\mu(\sigma)>0\}$ is
a tree which might be called the support of $\mu$.

Every tree-measure $\mu$ on $\{0,1\}^{\leq h}$ defines a distribution
on the nodes of the tree as follows: first choose a level $0\leq i\leq h$
uniformly, and then choose a node $\sigma\in\{0,1\}^{i}$ in level
$i$ with the probability given by $\mu$, i.e. $\mu(\sigma)$ (we
have already noted that at each level the masses sum to $1$). Thus
the probability of $A\subseteq T$ is
\[
\mathbb{P}_{\mu}(A)=\frac{1}{h+1}\sum_{\sigma\in A}\mu(\sigma)
\]
and the expectation of $f:\{0,1\}^{\leq h}\rightarrow\mathbb{R}$
is
\[
\mathbb{E}_{\mu}(f)=\frac{1}{h+1}\sum_{k=0}^{n}\sum_{\sigma\in\{0,1\}^{k}}\mu(\sigma)f(\sigma)
\]
Sometimes we write $\mathbb{P}_{\sigma\sim\mu}$ or $\mathbb{E}_{\sigma\sim\mu}$
to define $\sigma$ as a random node, as in the expression
\[
\mathbb{P}_{\sigma\sim\mu}(\sigma\in T\mbox{ and }\sigma\mbox{ has two children in }T)=\frac{1}{h+1}\sum_{\sigma\in T}\mu(\sigma)1_{\{\sigma0,\sigma1\in T\}}
\]

Given the tree $T$ of height $h$, it is natural to consider the
uniform probability measure $\mu_{\partial T}$ on $\partial T$ and,
as described above, extend it to a tree-measure, which we denote $\mu_{T}$.
In this case we abbreviate the probability and expectation operators
above by $\mathbb{P}_{T}$ and $\mathbb{E}_{T}$, etc. It is important
to note that \emph{choosing a node according to $\mu_{T}$ is not
the same as choosing a node uniformly from $T$}. The latter procedure
is usually heavily biased towards sampling from the leaves, since
these generally constitute a large fraction of the nodes (in the full
binary tree, sampling this way gives a leaf with probability $>1/2$).
In contrast, $\mu_{T}$ samples uniformly from the levels, and within
each level we sample according to the relative number of leaves descended
from each node.

Trees and tree-measures are naturally related to sets and measures
on $[0,1)$ using binary coding. Given a set $X\subseteq[0,1)$ and
$h\in\mathbb{N}$, we lift $X$ to a tree $T$ of height $h$ by taking
all the initial sequences of length $\leq h$ of binary expansions
of points in $X$, with the convention that the expansion terminates
in $1$s if there is an ambiguity. We remark that for~$k\leq h$,
\[
N_{1/2^{k}}(X)\leq\left|T\cap\{0,1\}^{k}\right|\leq2N_{1/2^{k}}(X)
\]
Similarly, a probability measure $\mu$ on $[0,1)$ can be lifted
to a tree-measure $\widetilde{\mu}$ on $\{0,1\}^{\leq h}$ by defining
$\widetilde{\mu}(\sigma)$ equal to the mass of the interval of numbers
whose binary expansion begins with $\sigma$.

\subsection{\label{sub:Inverse-theorems-for-power-growth}Inverse theorems in
the power-growth regime}

We need some terminology for describing the local structure of trees.
We say that $T$ has \emph{full branching} \emph{for $m$ generations
at $\sigma$} if $\sigma$ has all $2^{m}$ possible descendants $m$
generations below it, that is, $\sigma\eta\in T$ for all $\eta\in\{0,1\}^{m}$.
At the other extreme, we say that $T$ is \emph{fully concentrated}
\emph{for $m$ generations at $\sigma$ }if $\sigma$ has a single
descendant $m$ generations down, that is, there is a unique $\eta\in\{0,1\}^{m}$
with $\sigma\eta\in T$.

Let us return to the example $A_{n}$ from Section \ref{sub:Sumsets-with-power-growth}
and examine the associated tree $T_{n}$ of height $n^{2}$. For every
$i<n$, every node at level $i^{2}$ has full branching for $i$ generations;
and every node at level $i^{2}+i$ is fully concentrated for $i+1$
generations. Consequently, for every $j\in[i^{2},i^{2}+i)$ every
node of level $j$ has full branching for one generation; for $j\in[i^{2}+i,(i+1)^{2})$,
every node at level $j$ is fully concentrated for one generation.
We also have the following statement: For every $m$ we can partition
the levels $0,1,\ldots,n^{2}$ into three sets $U,V,W$, such that
(a) For every $i\in U$, every level-$i$ node has full branching
for $m$ generations; (b) For every $j\in V$, every level-$j$ node
is fully concentrated for $m$ generations; and (c) $W$ is a negligible
fraction of the levels, specifically $|W|/n^{2}=o(1)$ as $n\rightarrow\infty$
(with $m$ fixed). Of course, $U=\bigcup_{i>m}[i^{2},i^{2}+i-m)$,
$V=\bigcup_{i>m}[i^{2}+i,(i+1)^{2}-m)$, and $W$ is the set of remaining
levels. This is pictured schematically in Figure \ref{fig:A+A-small}.

\begin{figure}
\centering{}\includegraphics[scale=0.7]{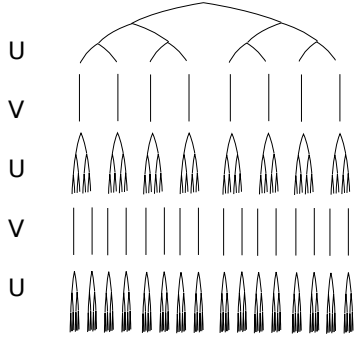}\caption{\label{fig:A+A-small}A tree with alternating levels having full branching
at some levels, full concentration at others, and a few levels omitted.
Schematically this is what the tree associated to $A_{n}$ in Section
\ref{sub:Sumsets-with-power-growth} looks like, as well as the conclusion
of Theorem \ref{thm:inverse-theorem-A+A} (with $W$ indicated by
the small space between levels). }
\end{figure}

Does this picture hold in general when $|A+A|\leq|A|^{1+\delta}$?
Certainly not exactly, since we can always pass to a subset $A'\subseteq A$
of size $|A'|\geq|A|^{1-\delta}$ and get a set with similar doubling
behavior (for a constant loss in $\delta$), but much less structure.
One can also perturb it in other ways. However, in a looser sense,
the picture above is quite general. One approach is to pass to a subtree
of reasonably large relative size. Such an approach was taken by Bourgain
in \cite{Bourgain2003,Bourgain2010}. The approach taken in \cite{Hochman2012a}
is more statistical, and in a sense it gives a description of the
entire tree, but requires us to weaken the notion of concentration.
Given $\delta>0$, we say that $T$ is $\delta$\emph{-concentrated}
for $m$ generations at $\sigma\in T$ if there exists $\eta\in\{0,1\}^{m}$
such that
\[
\mu_{T}(\sigma\eta)\geq(1-\delta)\mu_{T}(\sigma)
\]
where $\mu_{T}$ is the tree-measure associated to $T$. In other
words, $T$ is $\delta$-concentrated at $\sigma$ if it is possible
to remove an $\delta$-fraction of the leaves descended from $\sigma$
in such a way that the resulting tree becomes fully concentrated for
$m$ generations at $\sigma$. Note that this definition is not purely
local, since it depends not only on the depth-$m$ subtree of $T$
rooted at $\sigma$, but on the entire subtree rooted at $\sigma$,
since the weights on $S=\{\sigma\eta\,:\,\eta\in\{0,1\}^{m}\}$ are
determined by the number of leaves of $T$, not by $S$ itself. 
\begin{thm}
\label{thm:inverse-theorem-A+A}For every $\varepsilon>0$ and $m>1$,
there is a $\delta>0$ such that for all sufficiently small $\rho>0$
the following holds. Let $X\subseteq[0,1]$ be a finite set such that
\[
N_{\rho}(X+X)\leq N_{\rho}(X)^{1+\delta}
\]
and let $T$ be the associated tree of height $h=\left\lceil \log(1/\rho)\right\rceil $.
Then the levels $0,1,\ldots,h$ can be partitioned into sets $U,V,W$
such that
\begin{enumerate}
\item For every $i\in U$, 
\[
\mathbb{P}_{\sigma\sim T}(T\mbox{ has full branching at }\sigma\mbox{ for }m\mbox{ generations}\;|\;\sigma\mbox{ is in level }i)>1-\varepsilon.
\]

\item For every $j\in V$, 
\[
\mathbb{P}_{\sigma\sim T}(T\mbox{ is }\varepsilon\mbox{-concentrated at }\sigma\mbox{ for }m\mbox{ generations}\;|\;\sigma\mbox{ is in level }j)>1-\varepsilon.
\]

\item $|W|<\varepsilon h$.
\end{enumerate}
\end{thm}
If $X$ is $\rho$-separated, note that the hypothesis is essentially
the same as $|X+X|\leq|X|^{1+\varepsilon}$. 

Our analysis of self-similar sets requires the following asymmetric
variant, which is easily seen to imply the symmetric one above. To
motivate it, note that $|A+B|\leq C|A|$ can occur for two trivial
reasons: One is that $A=\{1,\ldots,n\}$ for some $n$ and $B\subseteq\{1,\ldots,n\}$
is arbitrary. The second is that $B=\{b\}$, a singleton, and $A$
is arbitrary. The following theorem says that when $|A+B|\leq|A|^{1+\delta}$
then there are essentially two kinds of scales: those where, locally,
the sets $A,B$ look like in the first trivial case, and those where,
locally, $A,B$ look like the second trivial case. See figure \ref{fig:A+B-small}.

\begin{figure}
\centering{}\includegraphics[scale=0.7]{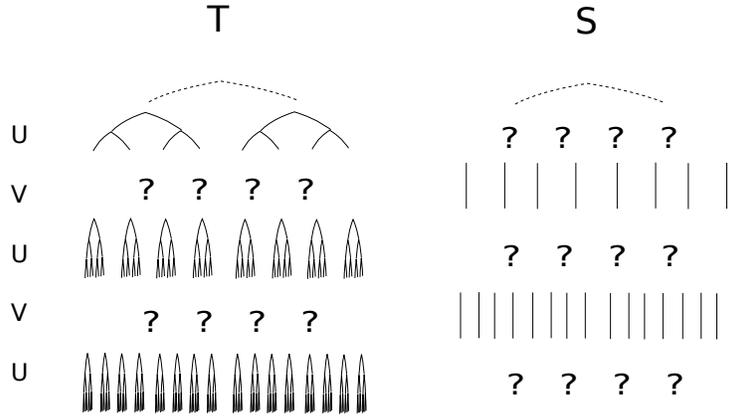}\caption{\label{fig:A+B-small}Schematic representation of the conclusion of
Theorem \ref{thm:A+B-small} (with $W$ indicated by the small space
between levels). }
\end{figure}

\begin{thm}
\label{thm:A+B-small}For every $\varepsilon>0$ and $m>1$, there
is a $\delta>0$ such that for all sufficiently small $\rho>0$ the
following holds. Let $X,Y\subseteq[0,1]$ be finite sets such that
\[
N_{\rho}(X+Y)\leq N_{\rho}(X)^{1+\delta}
\]
Let $T,S$ be the associated trees of height $h=\left\lceil \log(1/\rho)\right\rceil $,
respectively. Then the levels $0,1,\ldots,h$ can be partitioned into
sets $U,V,W$ such that
\begin{enumerate}
\item For every $i\in U$, 
\[
\mathbb{P}_{\sigma\sim T}(T\mbox{ has full branching at }\sigma\mbox{ for }m\mbox{ generations}\;|\;\sigma\mbox{ is in level }i)>1-\varepsilon
\]
(but we know nothing about $S$ at level $i$).
\item For every $j\in V$, 
\[
\mathbb{P}_{\sigma\sim S}(S\mbox{ is }\varepsilon\mbox{-concentrated at }\sigma\mbox{ for }m\mbox{ generations}\;|\;\sigma\mbox{ is in level }j)>1-\varepsilon
\]
(but we know nothing about $T$ at level $j$). 
\item $|W|<\varepsilon h$.
\end{enumerate}
\end{thm}
The theorems above follow from \cite[Theorems 2.7 and 2.9]{Hochman2012a},
using the fact that high enough entropy at a given scale implies full
branching, and small enough entropy at a given scale implies $\delta$-concentration.

\section{\label{sec:Main-reduction-heuristic}A conceptual proof of Theorem
\ref{thm:main} }

In this section we give a heuristic proof of Theorem \ref{thm:main}.
We begin with some general observations about self-similar sets. Then
we explain how the theorem is reduced to a statement about sumset
growth. Finally, we demonstrate how the inverse theorems of the previous
section are applied.

From now on let $\Phi=\{f_{i}\}_{i\in\Lambda}$ be an IFS with attractor
$X$, as in the introduction. We assume that $0\in X\subseteq[0,1)$;
this can always be achieved by a change of coordinates, which does
not affect the statement of Theorem \ref{thm:main}.

\subsection{\label{sub:Sumset-structure-of-SSSs}Sumset structure of self-similar
sets}

Our analysis will focus on finite approximations of $X$. Define the
$n$-th approximations by 
\[
X_{n}=\{f_{\underline{i}}(0)\,:\,\underline{i}\in\Lambda^{n}\}
\]
Clearly $X_{n}\subseteq X$. Also note that $|X_{n}|\leq|\Lambda|^{n}$,
with a strict inequality for some $n$ if and only if exact overlaps
occur. Self similarity enters our argument via the following lemma. 
\begin{lem}
\label{lem:sumset-structure-of-SSSs}For any $m,n\in\mathbb{N}$,
\begin{eqnarray}
X & = & X_{m}+r^{m}X\label{eq:cocycle-for-X}\\
X_{m+n} & = & X_{m}+r^{m}X_{n}\label{eq:cocycle-for-Xn}
\end{eqnarray}
\end{lem}
\begin{proof}
By (\ref{eq:iterated-recursion}) and (\ref{eq:iterated-f-b}), 
\begin{eqnarray*}
X & = & \bigcup_{\underline{i}\in\Lambda^{n}}f_{\underline{i}}(X)\\
 & = & \bigcup_{\underline{i}\in\Lambda^{n}}\{f_{\underline{i}}(0)+r^{m}x\,:\, x\in X\}\\
 & = & X_{m}+r^{m}X
\end{eqnarray*}
which is the first identity. To prove the second, for $\underline{i}\in\Lambda^{m}$
and $\underline{j}\in\Lambda^{n}$ denote by $\underline{i}\underline{j}$
their concatenation. By (\ref{eq:iterated-f-a}), 
\begin{eqnarray*}
f_{\underline{i}\underline{j}}(0) & = & \sum_{k=1}^{m}a_{i_{k}}r^{k-1}+r^{m}\sum_{k=1}^{n}a_{j_{k}}r^{k-1}\\
 & = & f_{\underline{i}}(0)+r^{m}f_{\underline{j}}(0)
\end{eqnarray*}
hence 
\begin{eqnarray*}
X_{m+n} & = & \{f_{\underline{i}\underline{j}}(0)\,:\,\underline{ij}\in\Lambda^{m+n}\}\\
 & = & \{f_{\underline{i}}(0)+r^{m}f_{\underline{j}}(0)\,:\,\underline{i}\in\Lambda^{m},\underline{j}\in\Lambda^{n}\}\\
 & = & X_{m}+r^{m}X_{n}\qedhere
\end{eqnarray*}

\end{proof}
Let us demonstrate the usefulness of this lemma by showing that $\bdim(X)$
exists. First, since $r^{m}X$ is of diameter $\leq r^{m}$, it is
easy to deduce from (\ref{eq:cocycle-for-X}) that $N_{r^{n}}(X_{n})$,
$N_{r^{n}}(X)$ differ by at most a factor of $2$. Thus the existence
of $\bdim X$ is equivalent to existence of the limit $\frac{1}{m}\log N_{r^{m}}(X_{m})$
as $n\rightarrow\infty$. Next, we have a combinatorial lemma.
\begin{lem}
\label{lem:iterated-covering-number-of-sumset}Let $A,B\subseteq\mathbb{R}$
with $B\subseteq[0,\varepsilon)$. Then for any $\gamma<\varepsilon$,
\[
N_{\gamma}(A+B)\geq\frac{1}{3}\cdot N_{\varepsilon}(A)\cdot N_{\gamma}(B)
\]
\end{lem}
\begin{proof}
Let $\mathcal{I}=\{I_{i}\}_{i=1}^{N_{\varepsilon}(A)}$ be an optimal
cover of $A$ by disjoint intervals of length $\varepsilon$. Let
$\mathcal{J}=\{J_{j}\}_{j=1}^{N_{\gamma}(A+B)}$ be an optimal cover
of $A+B$ by intervals of length $\gamma$. For each $I_{i}\in\mathcal{I}$
fix a point $a_{i}\in A\cap I_{i}$ and note that $a_{i}+B\subseteq A+B$
is covered by $\mathcal{J}$, so $a_{i}+B$ intersects at least $N_{\gamma}(B)$
intervals in $\mathcal{J}$. If each interval $J_{j}$ intersects
a unique translate $a_{i}+B$, we would conclude that $N_{\gamma}(A+B)\geq N_{\varepsilon}(A)N_{\gamma}(B)$.
While $a_{i}$ may not be unique, we can argue as follows: Since $B\subseteq[0,\varepsilon)$,
if $J_{j}=[u,u+\varepsilon]$ and intersects $a+B$ for some $a\in A$,
then $a\in[u-\varepsilon,u+2\varepsilon)$. Since the intervals $I_{i}$
are disjoint and of length $\varepsilon$, there are most $3$ intervals
$I_{i}\in\mathcal{I}$ that $a$ could belong to. The claim follows. 
\end{proof}
Since $X\subseteq[0,1)$ we have $r^{m}X_{n}\subseteq[0,r^{m})$,
so by the lemma,
\begin{eqnarray*}
N_{r^{m+n}}(X_{m+n}) & = & N_{r^{m+n}}(X_{m}+r^{m}X_{n})\\
 & \geq & \frac{1}{3}\cdot N_{r^{m}}(X_{m})\cdot N_{r^{m+n}}(r^{m}X_{n})\\
 & = & \frac{1}{3}\cdot N_{r^{m}}(X_{m})\cdot N_{r^{n}}(X_{n})
\end{eqnarray*}
where in the last equality we used the identity $N_{t\varepsilon}(tZ)=N_{\varepsilon}(Z)$.
Taking logarithms, this shows that the sequence $s_{n}=\log N_{r^{m}}(X_{m})$
is approximately super-additive in the sense that $s_{m+n}\geq s_{m}+s_{n}-C$
for a constant $C$. The existence of the limit of $\frac{1}{n}s_{n}$
as $n\rightarrow\infty$ is then well known (perhaps it is better
known when $C=0$ and $s_{n}$ is (really) super-additive. The proof
for $C=0$ works also in the $C>0$ case; alternatively, note that
$s'_{n}=s_{n}-\log n$ becomes super-additive after excluding finitely
many terms, so $\lim\frac{1}{n}s'_{n}$ exists, and $\frac{1}{n}s'_{n}-\frac{1}{n}s_{n}\rightarrow0$).

\subsection{\label{sub:Main-reduction-first-try}From Theorem \ref{thm:main}
to additive combinatorics }

Let us return to our main objective, Theorem \ref{thm:main}. Continuing
with the previous notation, write 
\begin{eqnarray*}
\alpha & = & \bdim X\\
\beta & = & \min\{1,\sdim X\}
\end{eqnarray*}
and suppose, by way of contradiction, that $\alpha<\beta$ and that
for some $k\in\mathbb{N}$ we have $\Delta_{n}\geq2^{-kn}$ for all
$n$ (in particular, there are no exact overlaps). We make a number
of observations. The first is rather trivial: that ``too small''
dimension means that there are intervals of length $r^{m}$ containing
exponentially many points from $X_{m}$. Precisely,
\begin{prop}
\label{prop:fiber-covering-numbers}Let $\sigma=\frac{1}{2}(\beta-\alpha)>0$.
Then for every large enough $m$, there is an interval $I_{m}$ of
length $r^{m}$ such that $|X_{m}\cap I_{m}|>r^{-\sigma m}$.\end{prop}
\begin{proof}
As we have already noted, $\frac{1}{m\log(1/r)}\log N_{r^{m}}(X_{m})\rightarrow\alpha$
as $m\rightarrow\infty$. Thus for large enough $m$, 
\[
N_{r^{m}}(X_{m})<r^{-(\beta-\sigma)m}
\]
On the other hand, since there are no exact overlaps, 
\[
|X_{m}|=|\Lambda|^{m}=r^{-m\sdim(X)}\geq r^{-m\beta}
\]
Thus in an optimal cover of $X_{m}$ by $r^{m}$-intervals, at least
one must contain $|X_{m}|/N_{r^{m}}(X_{m})\geq(1/r)^{\sigma m}$ points.
\end{proof}
We now wish to extract more information from the sumset identity $X_{m+n}=X_{m}+r^{m}X_{n}$.
In itself it provides limited information about the covering number
$N_{m+n}(X_{n})$, since the summands live at different scales. This
is what was used earlier in proving super-additivity of $s_{n}=\log N_{r^{m}}(X_{m})$.
The next step is to localize the sumset relation. 
\begin{prop}
\label{prop:local-sumset}For all $\delta>0$, for all large $m$
there exists an interval $J_{m}$ of length $r^{m}$ such that $X_{m}\cap J_{m}\neq\emptyset$
and, writing $n=km$, 
\begin{equation}
N_{r^{m+n}}((X_{m}\cap J_{m})+r^{m}X_{n})<r^{-(1+\delta)\alpha n}\label{eq:local-sumset}
\end{equation}
\end{prop}
\begin{proof}
Fix $m$, set $n=km$, and let $\mathcal{J}$ denote the partition
of $\mathbb{R}$ into intervals $[ur^{m},(u+1)r^{m})$, $u\in\mathbb{Z}$,
whose lengths are $r^{m}$. Since $X_{m}=\bigcup_{J\in\mathcal{J}}(X_{m}\cap J)$,
we can re-write (\ref{eq:cocycle-for-Xn}) as 
\begin{eqnarray}
X_{m+n} & = & X_{m}+r^{m}X_{n}\nonumber \\
 & = & \bigcup_{J\in\mathcal{J}}((X_{m}\cap J)+r^{m}X_{n})\label{eq:local-cocycle-for-Xn}
\end{eqnarray}
Since $X_{m}\cap J\subseteq[ur^{m},(u+1)r^{m})$ for some $u$ and
$r^{m}X_{n}\subseteq[0,r^{m})$, we have $(X_{m}\cap J)+r^{m}X_{n}\subseteq[ur^{m},(u+2)r^{m})$
and in particular each set in the union (\ref{eq:local-cocycle-for-Xn})
is of diameter $\leq2r^{m}$. On the other hand, no interval of length
$r^{m+n}$ intersects more than three of the sets $[ur^{m},(u+2)r^{m})$.
Therefore, arguing as in the proof of Lemma \ref{lem:iterated-covering-number-of-sumset},
\begin{eqnarray*}
N_{r^{m+n}}(X_{m+n}) & \geq & \frac{1}{3}\cdot N_{r^{m}}(X_{m})\cdot\min_{J\in\mathcal{J}\,:\, X_{m}\cap J\neq\emptyset}N_{r^{m+n}}((X_{m+n}\cap J)+r^{m}X_{n})
\end{eqnarray*}
so
\begin{eqnarray*}
\min_{J\in\mathcal{J}\,:\, X_{m}\cap J\neq\emptyset}N_{r^{m+n}}((X_{m+n}\cap J)+r^{m}X_{n}) & \leq & 3\cdot\frac{N_{r^{m+n}}(X_{m+n})}{N_{r^{m}}(X_{m})}\\
 & \leq & 3\cdot\frac{r^{-(\alpha+o(1))(m+n)}}{r^{-(\alpha+o(1))m}}\\
 & = & r^{-(\alpha+o(1))n}\qquad\mbox{as }m\rightarrow\infty
\end{eqnarray*}
The proposition follows.
\end{proof}
Now suppose that it so happens that, for large $m$, the propositions
above produce the \emph{same }interval: $I_{m}=J_{m}$. We then we
would have the following: 

\begin{fprop} \label{prop:small-sumset}Suppose that $\dim X<\min\{1,\sdim X\}$
and $\Delta_{n}\geq2^{-kn}$ for all $n$. Then there is a constant
$\tau>0$ such that, for every $\delta>0$ and all suitably large
$n$, there is a subset $Y_{n}\subseteq[0,1]$ with 
\begin{eqnarray}
N_{r^{n}}(Y_{n}) & \geq & 2^{\tau n}\label{eq:Yn-small}\\
N_{r^{n}}(X_{n}+Y_{n}) & \leq & N_{r^{n}}(X_{n})^{1+\delta}\label{eq:Xn+Yn-small}
\end{eqnarray}
\end{fprop}
\begin{proof}
[``Proof''] Let $\sigma$ be as in Proposition \ref{prop:fiber-covering-numbers}
and take $\tau=\sigma/(k\log(1/r))$. As before write $n=(k+1)m$,
and assume that the intervals $I_{m},J_{m}$ provided by the two previous
propositions coincide for arbitrarily large $m$: $I_{m}=J_{m}=[a_{m},b_{m})$.
Let
\[
Y_{m}=r^{-m}(X_{m}\cap I_{m})
\]
and note that $Y_{m}\subseteq[0,1)$. Now, by choice of $I_{m}$ we
know that $|X_{m}\cap I_{m}|\geq r^{-\sigma m}$, and since $\Delta_{m}\geq2^{-km}=2^{-n}$,
we know that every two points in $X_{m}\cap I_{m}$ are separated
by at least $2^{-n}$. Therefore, 
\begin{eqnarray*}
N_{r^{n}}(Y_{m}) & = & N_{r^{m+n}}(X_{m}\cap I_{m})\\
 & \geq & r^{-\sigma m}
\end{eqnarray*}
Using the identity $N_{t\varepsilon}(tZ)=N_{\varepsilon}(Z)$ with
$t=r^{m}$ and $Z=Y_{m}$, we conclude that
\[
N_{r^{-n}}(Y_{m})\geq r^{-\sigma m}=r^{\tau n}
\]
Similarly, since $X_{n}+Y_{m}=r^{-m}((X_{m}\cap I_{m})+r^{m}X_{n})$,
from the definition of $J_{m}$ and the identity $N_{t\varepsilon}(tZ)=N_{\varepsilon}(Z)$
again, we find that for large enough $n$ (equivalently, $m$),
\begin{eqnarray*}
N_{r^{n}}(X_{n}+Y_{n}) & = & N_{r^{m+n}}((X_{m}\cap I_{m})+r^{m}X_{n})\\
 & \leq & r^{-(1+\delta)\alpha n}\\
 & \leq & N_{r^{n}}(X_{n})^{(1+2\delta)}
\end{eqnarray*}
where in the last inequality we again used the fact that $N_{r^{m}}(X_{m})\sim r^{-n\alpha}$.
\end{proof}
The task of showing that the conclusion of the ``Proposition'' is
impossible falls within the scope of additive combinatorics. Heuristically,
it cannot happen because, being a fractal, $X_{n}$ has very little
``additive structure''. This intuition is correct, as we discuss
in the next section.

But can one really ensure that $I_{m}$,$J_{m}$ coincide? A natural
attempt would be to show that, for a fixed optimal $r^{m}$-cover
of $X_{m}$, ``most'' intervals of length $r^{m}$ can play each
of the roles, and hence a positive fraction can play both. In fact,
for every $\eta>0$, for large $m$ at least a $(1-\eta)$-proportion
of these intervals will be a good choice for $J_{m}$. Unfortunately,
although the number of candidates for $I_{m}$ can be shown to be
exponential in $m$, it could still be exponentially small compared
to $N_{r^{m}}(X_{m})$, and so we cannot conclude that the two families
of ``good'' intervals have members in common. It is possible that
more sophisticated counting can make this work, but the approach that
is currently simplest is to replace covering numbers by the entropy,
at an appropriate scale, of the uniform measure on $X_{m}$. We return
to this in Section \ref{sec:From-Proof-to-Proof}.

\subsection{\label{sub:Getting-a-contradiction}Getting a contradiction }

Our goal now is to demonstrate that the conclusion of ``Proposition''
\ref{prop:small-sumset} is impossible. The argument we give again
falls short of this goal, but it gives the essential ideas of the
proof. Thus, we ask the reader to suspend his disbelief a little longer.

Let $\tau>0$ be as given in ``Proposition'' \ref{prop:small-sumset}.
Choose a very small parameter $\varepsilon>0$ which we shall later
assume is small compared to $\tau$. Choose $m$ large enough that
\[
N_{r^{m}}(X_{m})\geq r^{-m(1-\varepsilon)\alpha}
\]

Apply the inverse theorem \ref{thm:A+B-small} with parameters $\varepsilon,m$
and obtain the promised $\delta>0$. From ``Proposition'' \ref{prop:small-sumset}
obtain the corresponding $Y_{n}\subseteq[0,1)$ satisfying (\ref{eq:Yn-small})
and (\ref{eq:Xn+Yn-small}). 

Write $T^{n}$ for the tree of height $h_{n}=[1/r^{n}]$ associated
to $X_{n}$ and $S^{n}$ for the tree of the same height associated
to $Y_{n}$. From our choice of $\delta$ and (\ref{eq:Xn+Yn-small}),
by the inverse theorem there is a partition $U_{n}\cup V_{n}\cup W_{n}$
of $\{1,\ldots,h_{n}\}$ such that
\begin{enumerate}
\item [(I)] At scales $i\in U_{n}$, a $1-\varepsilon$ fraction of nodes
of $T^{n}$ at level $i$ have full branching for $m$-generations.
\item [(II)] At scales $j\in V_{n}$, a $1-\varepsilon$ fraction of nodes
of $S^{n}$ at level $j$ are $\varepsilon$-concentrated for $m$
generations. 
\item [(III)] $|W_{n}|\leq\varepsilon h_{n}$.
\end{enumerate}
Our first task is to show that $V_{n}$ is not too large. It is quite
clear (or at least believable) that if a tree has few nodes with more
than one child, then it can have only an exponentially small number
of leaves. The same is true if we only assume, for a small $\lambda>0$,
that most nodes are $\lambda$-concentrated. More precisely,
\begin{lem}
\label{lem:concentration-implies-small-tree}Let $S$ be a tree of
height $h$, let $\lambda>0$ and $\ell\geq1$. Suppose that 
\[
\mathbb{P}_{S}(\sigma\in S\,:\, S\mbox{ is }\lambda\mbox{-concentrated at }\sigma\mbox{ for }\ell\mbox{ generations}\}>1-\lambda
\]
Then $|\partial S|\leq2^{\lambda'\cdot h}$ where $\lambda'\rightarrow0$
as $\lambda\rightarrow0$ and $h/\ell\rightarrow\infty$.
\end{lem}
We leave the proof to the motivated reader. We note that this lemma
is superseded by Proposition \ref{prop:entorpy-averages}, which gives
a stronger statement and has a simpler proof.

We apply the lemma to $S=S^{n}$ with $\ell=m$. Choose $\lambda$
small enough that $\lambda'<\tau$ for large $n$ (hence large $h_{n}$).
Thus $\lambda$ depends only on $\tau$ and we may assume that at
the start we chose $\varepsilon<\frac{1}{2}\lambda$. Suppose that
we had $|V_{n}|>(1-\lambda/2)h_{n}$. Since in each level $j\in V_{n}$
a $(1-\varepsilon)$-fraction of the nodes (with respect to the tree
measure $\mu_{S^{n}}$) is $\varepsilon$-concentrated, at least the
same fraction is $\lambda$-concentrated, and we would conclude 
\begin{eqnarray*}
\mathbb{P}_{S^{n}}(\sigma\in S^{n}\,:\, S^{n}\mbox{ is }\lambda\mbox{-concentrated at }\sigma\mbox{ for }m\mbox{ generations}) & \geq & \frac{1}{h_{n}}|V_{n}|\cdot(1-\varepsilon)\\
 & > & (1-\frac{\lambda}{2})(1-\varepsilon)\\
 & > & 1-\lambda
\end{eqnarray*}
From the lemma we would have $N_{r^{n}}(Y_{n})\leq|\partial S^{n}|<2^{\lambda'h_{n}}<2^{\tau h_{n}}$,
contradicting (\ref{eq:Yn-small}). Thus, we conclude that
\[
|V_{n}|<(1-\frac{\lambda}{2})h_{n}
\]
Consequently, assuming as we may that $\varepsilon<\lambda/6$,
\begin{equation}
|U_{n}|=h_{n}-|V_{n}|-|W_{n}|\geq(\frac{\lambda}{2}-\varepsilon)h_{n}>\frac{\lambda}{3}h_{n}\label{eq:Un-is-large}
\end{equation}

So far we have seen that $U_{n}$ consists of a positive fraction
of the levels of $T^{n}$, and hence a positive fraction of nodes
in $T^{n}$ have full branching for $m$ generations. Our next task
will be to show that most of the remaining nodes have roughly $r^{-\alpha m}$
descendants $m$ generations down. This is where we use self-similarity
again in an essential way. 
\begin{prop}
\label{prop:SSSs-have-uniform-branching}If $m$ is large enough,
then for all large enough $n$, 
\[
\mathbb{P}_{\sigma\sim\mu_{T^{n}}}\left(\begin{array}{c}
\sigma\mbox{ has }\geq2^{(1-\varepsilon)\alpha m}\mbox{ descendants }\\
m\mbox{ generations down in }T
\end{array}\right)>1-\varepsilon
\]
\end{prop}
\begin{proof}
[Proof sketch] A node $\sigma\in T^{n}$ of level $\ell$ corresponds
to an interval $I=[\frac{u}{2^{\ell}},\frac{u+1}{2^{\ell}})$. We
call such intervals level-$\ell$ intervals, and recall that the probability
induced from $\mu_{T^{n}}$ on level-$\ell$ intervals is just proportional
to $|I\cap X_{n}|$. The claim is then that if we choose $0\leq\ell\leq h_{n}$
uniformly and then choose a level-$\ell$ interval $I$ at random,
then with probability at least $1-\varepsilon,$ we will have $N_{2^{-\ell-m}}(I\cap X_{n})\geq2^{-(1-\varepsilon)\alpha m}$.
In order to prove this, it is enough to show that for all levels $0\leq\ell\leq(1-\frac{\varepsilon}{2})h_{n}$,
if we choose a level-$\ell$ interval $I$ at random, then with probability
at least $1-\frac{\varepsilon}{2}$ we have $N_{2^{-\ell-m}}(I\cap X_{n})\geq2^{-(1-\varepsilon)\alpha m}$.

Fix a parameter $m_{0}$ depending on $\varepsilon$ and assume $m,n$
large with respect to it. Observe that $X_{n}$ decomposes into a
union of copies of $X_{n'}$ scaled by approximately $2^{-\ell-m_{0}}$.
More precisely, choosing $u\in\mathbb{N}$ such that $r^{u}\approx2^{-\ell-m_{0}}$,
by (\ref{eq:cocycle-for-Xn}) we have 
\[
X=X_{u}+r^{u}X_{n-u}=\bigcup_{x\in X_{u}}(x+r^{u}X_{n-u})
\]
The idea is now the following. The translates $x+r^{u}X_{n-u}$ in
the union are of diameter $r^{u}\approx2^{-\ell}/2^{m_{0}}$, which
is much smaller than $2^{-\ell}$, and hence with probability at least
$1-\frac{\varepsilon}{2}$ a level-$\ell$ interval $I$ will contain
an entire translate $x+r^{u}X_{n-u}$ from the union above. The details
of the proof are somewhat tedious and we omit them. The point is that,
if $x+r^{u}X_{n-u}\subseteq I$, and assuming that $m$ is large enough
relative to $\varepsilon,m_{0}$, we have 
\begin{eqnarray*}
N_{2^{-\ell-m}}(I\cap X_{n}) & \geq & N_{2^{-\ell-m}}(x+r^{u}X_{n-u})\\
 & = & N_{2^{-\ell-m}r^{-u}}(X_{n-u})\\
 & \approx & N_{2^{-(m-m_{0})}}(X_{n-u})\\
 & > & 2^{(1-\varepsilon)\alpha m}
\end{eqnarray*}
which is what we wanted to prove.
\end{proof}
Now that we know that most nodes in $T^{n}$ have many descendants,
and a positive fraction have the maximal possible number of descendants,
$m$ generations down, the last ingredient we need is a way to use
this information to get a lower bound on the number of leaves in $T^{n}$.
Heuristically, this is the analog of the upper bound we had in Lemma
\ref{lem:concentration-implies-small-tree}.

\begin{fprop} \label{prop:Large-branching-implies-large-tree}Let
$T$ be a tree of height $h$, let $m\geq0$, and suppose that the
nodes of $T$ can be partitioned into disjoint sets $A_{1},\ldots,A_{\ell}$
such that each node $\sigma\in A_{i}$ has $2^{c_{i}m}$ descendants
$m$ generations down. Write $p_{i}=\mathbb{P}_{\mu_{T}}(A_{i})$.
Then
\[
|\partial T|\geq\prod_{i=1}^{\ell}2^{c_{i}\cdot p_{i}h}
\]

\end{fprop}

This ``Proposition'' is, unfortunately, incorrect, and the reader
may find it instructive to look for a counterexample. The statement
could be fixed if we made stronger assumptions than just bounding
the branching in each of the sets $A_{i}$, but the resulting argument
would almost certainly be more complicated than the proof in \cite{Hochman2012a},
and we do not pursue it. The correct statement is given in Proposition
\ref{prop:entorpy-averages} below.

We can now put the pieces together. By the defining property (I) of
$U_{n}$ and equation (\ref{eq:Un-is-large}), the set $A_{1}^{n}\subseteq T^{n}$
of nodes with full branching for $m$-generations satisfies 
\begin{eqnarray*}
\mathbb{P}_{T^{n}}(A_{1}^{n}) & \geq & \frac{1}{h_{n}}|U_{n}|\cdot(1-\varepsilon)\\
 & \geq & \frac{\lambda}{3}(1-\varepsilon)\\
 & \geq & \frac{\lambda}{4}
\end{eqnarray*}
assuming again $\varepsilon$ small compared to $\lambda$ (equivalently
$\tau$). Let $A_{2}^{n}$ denote the set of nodes of $T^{n}\setminus A_{1}^{n}$
which do \emph{not} have at least $2^{m(1-2\varepsilon)\dim X}$ descendants
$m$ generations down; by Proposition \ref{prop:SSSs-have-uniform-branching},
\[
\mathbb{P}_{T^{n}}(A_{2}^{n})<\varepsilon
\]
Therefore if we define $A_{3}^{n}=T^{n}\setminus\{A_{1}^{n}\cup A_{2}^{n}\}$
then all nodes in $A_{3}^{n}$ have at least $2^{m(1-2\varepsilon)\dim X}$
descendants $m$ generations down and
\[
\mathbb{P}_{T^{n}}(A_{3}^{n})=1-\mathbb{P}_{T^{n}}(A_{1}^{n})-\mathbb{P}_{T^{n}}(A_{2}^{n})
\]
In the terminology of the ``Proposition'', we have $p_{1}\geq\lambda/4$
and $p_{2}<\varepsilon$, hence $p_{3}\geq1-p_{1}-\varepsilon$. Also
$c_{1}=1$, $c_{3}=(1-2\varepsilon)\dim X$ and by default $c_{2}\geq0$.
From the ``Proposition'' we find that
\begin{eqnarray*}
|\partial T^{n}| & \geq & 2^{p_{1}h_{n}}\cdot2^{p_{2}\cdot c_{2}h_{n}}\cdot2^{p_{3}(1-2\varepsilon)\dim X\cdot h_{n}}\\
 & \geq & 2^{p_{1}h_{n}+(1-2\varepsilon)\dim X\cdot(1-p_{1}-\varepsilon)h_{n}}\\
 & \geq & 2^{(\dim X+\varepsilon)h_{n}}
\end{eqnarray*}
where in the last inequality we assumed that $\varepsilon$ is small
compared to $p_{1}$ (eqiovalently $\lambda$) and $\tau$. Since
$N_{r^{n}}(X)=|\partial T^{n}|^{1+o(1)}$ as $n\rightarrow\infty$,
this contradicts the definition of $\dim X$.

\subsection{Sums with self-similar sets}

What we ``proved'' above is the following statement which is of
independent interest, and is, moreover, true (a proof follows easily
from the methods of \cite{Hochman2012a}). 
\begin{thm}
For every any self-similar set $X$ with $\dim X<1$ and every $\tau>0$
there is a $\delta>0$ such that for all small enough $\rho>0$ and
any set $Y\subseteq\mathbb{R}$,
\[
N_{\rho}(Y)>(1/\rho)^{\tau}\qquad\implies\qquad N_{\rho}(X+Y)>N_{\rho}(X)^{1+\delta}
\]

\end{thm}
There is also a fractal version for Hausdorff dimension:
\begin{thm}
For every any self-similar set $X$ with $\dim X<1$ and every $\tau>0$
there is a $\delta>0$ such that for any set $Y\subseteq\mathbb{R}$,
\[
\dim Y>\tau\qquad\implies\qquad\dim(X+Y)>\dim X+\delta
\]

\end{thm}
For box dimension (lower or upper) the analogous statement follows
directly from the previous theorem. The version for Hausdorff dimension
requires slightly more effort and will appear in \cite{Hochman2013a}
along with the analog for measures.

\section{\label{sec:From-Proof-to-Proof}Entropy}

In this final section we discuss how to turn the outline above into
a valid proof. The main change is to replace sets by measures and
covering numbers by entropy. Each of the three parts of the argument
(inverse theorem, reduction to a statement about sumsets, and the
analysis of the sums) has an entropy analog which we indicate below,
along with a reference to the relevant part of \cite{Hochman2012a}.

The reader should note that the outline given below is designed to
match as closely as possible the argument from the previous section,
rather than the proof from \cite{Hochman2012a}. Although the ideas
and many of the details are the same, the original proof is direct,
whereas the one here is by contradiction. For this reason not all
of the statements below have exact analogs in \cite{Hochman2012a}.

\subsection{\label{sub:Entropy}Entropy }

We assume that the reader is familiar with the basic properties of
Shannon entropy, see for example \cite{CoverThomas06}. Let $\mathcal{I}_{\varepsilon}=\{[k\varepsilon,(k+1)\varepsilon)\}_{k\in\mathbb{Z}}$,
which is a partition of $\mathbb{R}$ into intervals of length $\varepsilon$.
The entropy $H(\mu,\mathcal{I}_{\varepsilon})$ of $\mu$ at scale
$\varepsilon$ is the natural measure-analog of the covering number
$N_{\varepsilon}(X)$, albeit in a logarithmic scale. For a measure
$\nu$ supported on a set $X$, the two quantities are related by
the basic inequality 
\[
0\leq H(\nu,\mathcal{I}_{\varepsilon})\leq\log\#\{I\in\mathcal{I}_{\varepsilon}\,:\, X\cap I\neq\emptyset\}\leq\log N_{\varepsilon}(X)+O(1)
\]
(the $O(1)$ error is because we are choosing a sub-cover of $X$
from a fixed cover of $\mathbb{R}$ rather than allowing arbitrary
$\varepsilon$-intervals). We introduce the normalized $\varepsilon$-scale
entropy: 
\[
H_{\varepsilon}(\nu)=\frac{1}{\log(1/\varepsilon)}H(\nu,\mathcal{I}_{\varepsilon})
\]
Thus, for $\nu$ supported on a set $X$ with well-defined box dimension,
the previous inequality implies 
\begin{equation}
\limsup_{\varepsilon\rightarrow0}H_{\varepsilon}(\nu)\leq\bdim X\label{eq:entropy-vs-dimension}
\end{equation}

\subsection{\label{sub:Inverse-theorems-for-entropy}Inverse theorems for entropy}

The measure-analog of the sumset operation is convolution, which for
discrete probability measures $\mu=\sum p_{i}\delta_{x_{i}}$ and
$\nu=\sum q_{j}\delta_{y_{j}}$ is%
\footnote{In general there is a similar formula: $\mu*\nu=\int\int\delta_{x+y}d\mu(x)d\nu(y)$,
where the integral is interpreted as a measure by integrating against
Borel functions.%
} 
\[
\mu*\nu=\sum_{i,j}p_{i}q_{j}\delta_{x_{i}+y_{j}}
\]
The entropy-analog of the small doubling condition $|A+A|\leq C|A|$
is the inequality $H(\mu*\mu)\leq H(\mu)+C'$, where $H(\mu)$ is
the entropy of a \,measure with respect to the partition into points
(remember that entropy is like cardinality, but in logarithmic scale).
Alternatively we could discretize at scale $\varepsilon$, giving
$H_{\rho}(\mu*\mu)\leq H_{\rho}(\mu)+O(1/\log(1/\varepsilon))$. Tao
\cite{Tao2010} has shown that such inequalities have implications
similar to Freiman's theorem. Related results were also obtained by
Madiman \cite{Madiman2008}, see also \cite{MadimanMarcusTetali2012}.

The regime that interests us is, as before, the analog of $|A+B|\leq|A|^{1+\delta}$,
which by formal analogy takes the form $H_{\rho}(\mu*\nu)\leq(1+\delta)H_{\rho}(\mu)$.
When $\mu$ is supported on $[0,1]$ we have $H_{\rho}(\mu)\leq1+o(1)$
as $\rho\rightarrow0$, and this inequality is implied (and in the
cases that interest us essentially equivalent to) 
\begin{equation}
H_{\rho}(\mu*\nu)\leq H_{\rho}(\mu)+\delta\label{eq:entropy-small-growth}
\end{equation}

Before stating the inverse theorem for entropy we need a few more
definitions. Consider the lift of $\mu$ to a tree-measure $\widetilde{\mu}$
on the full binary tree of height $h$ (see Section \ref{sub:Trees-and-tree-measures}).
Given a node $\sigma=\sigma_{1}\ldots\sigma_{k}$ and $m\in\mathbb{N}$,
write $\sigma\{0,1\}^{m}$ for the set of descendants of $\sigma$
$m$-generations down. Let $\widetilde{\mu}_{\sigma,m}$ denote the
probability measure on $\sigma\{0,1\}^{m}$ that assigns to each node
its normalized weight according to $\widetilde{\mu}$. Since $\sum_{\eta\in\sigma\{0,1\}^{m}}\widetilde{\mu}(\eta)=\widetilde{\mu}(\sigma)$,
this measure is given by $\widetilde{\mu}_{\sigma,m}(\eta)=\widetilde{\mu}(\eta)/\widetilde{\mu}(\sigma)$. 

We say that $\widetilde{\mu}$ is $\delta$-concentrated at $\sigma$
for $m$ generations if $H(\widetilde{\mu}_{\sigma,m})<\delta$, that
is, if $-\frac{1}{m}\sum_{\eta\in\sigma\{0,1\}^{m}}\frac{\widetilde{\mu}(\eta)}{\widetilde{\mu}(\sigma)}\log\frac{\widetilde{\mu}(\eta)}{\widetilde{\mu}(\sigma)}<\delta$.
For a tree measure $\mu_{T}$ associated to a tree $T$ and for fixed
$m$, this is equivalent to $T$ being $\delta'$-concentrated for
$m$ generations at $\sigma$ for an appropriate $\delta'$ which
tends to $0$ together with $\delta$. We say that $\widetilde{\mu}$
is $\delta$-uniform at $\sigma$ for $m$ generations if $H(\widetilde{\mu}_{\sigma,m})>\log m-\delta$.
Note that for $m$ fixed, when $\delta$ is small enough this implies
that $\widetilde{\mu}(\eta)>0$ for all $\eta\in\sigma\{0,1\}^{m}$,
so this indeed generalizes full branching. We can now state the inverse
theorem:
\begin{thm}
[Theorem 2.7 of \cite{Hochman2012a}] \label{thm:inverse-thm-entropy}For
every $\varepsilon>0$ and $m\geq1$, there is a $\delta>0$ such
that for sufficiently small $\rho>0$ the following holds. Let $\mu,\nu$
be probability measures on $[0,1]$ and suppose that
\[
H_{\rho}(\mu*\nu)\leq H_{\rho}(\mu)+\delta
\]
Let $\widetilde{\mu},\widetilde{\nu}$ denote the lifts of $\mu,\nu$
to the full binary trees of height $h=\left\lceil \log_{2}(1/\rho)\right\rceil $.
Then there is a partition of the levels $\{0,\ldots,h\}$ into three
sets $U\cup V\cup W$ such that
\begin{enumerate}
\item For $i\in U$,
\[
\mathbb{P}_{\sigma\sim\widetilde{\mu}}(\widetilde{\mu}\mbox{ is }\varepsilon\mbox{-uniform at }\sigma\mbox{ for }m\mbox{ generations}\;|\;\sigma\mbox{ is in level }i)>1-\varepsilon
\]

\item For $j\in V$,
\[
\mathbb{P}_{\sigma\sim\widetilde{\mu}}(\widetilde{\nu}\mbox{ is }\varepsilon\mbox{-concentrated at }\sigma\mbox{ for }m\mbox{ generations}\;|\;\sigma\mbox{ is in level }i)>1-\varepsilon
\]

\item $|W|<\delta n$.
\end{enumerate}
\end{thm}

\subsection{\label{sub:Reduction-to-convolution}Reduction of Theorem \ref{thm:main}
to a convolution inequality}

We return to our IFS $\Phi$ with attractor $0\in X\subseteq[0,1]$,
as in Section \ref{sec:Main-reduction-heuristic}. Define measures
$\mu^{(n)}$ analogous to $X_{n}$ by 
\[
\mu^{(n)}=\frac{1}{|\Lambda|^{n}}\sum_{\underline{i}\in\Lambda^{n}}\delta_{f_{\underline{i}}(0)}
\]
Write $S_{t}\mu(A)=\mu(t^{-1}A)$ (this is the usual push-forward
of $\mu$ by $S_{t}$). Then the analog of the sumset relation $X_{m+n}=X_{m}+r^{m}X_{n}$
is 
\[
\mu^{(m+n)}=\mu^{(m)}*S_{r^{m}}\mu^{(n)}
\]
The derivation is elementary, using the definition of convolution,
equation (\ref{eq:iterated-f-b}) and the identity $S_{t}\delta_{y}=\delta_{ty}$.
Next, as in Section \ref{sub:Sumset-structure-of-SSSs}, if we define
$s_{m}=H(\mu^{(m)},\mathcal{I}_{r^{m}})$ then the sequence $s_{n}$
is almost super-additive in the sense that $s_{m+n}\geq s_{m}+s_{n}-O(1)$.
This is proved by a similar argument to the covering number case but
in the language of entropy. It follows that the limit 
\[
\alpha=\lim_{m\rightarrow\infty}H_{r^{m}}(\mu^{(m)})
\]
exists. Since $\mu^{(m)}$ is supported on $X$, by (\ref{eq:entropy-vs-dimension})
we have $\alpha\leq\dim X$.

Turning to Theorem \ref{thm:main}, write 
\[
\beta=\min\{1,\sdim X\}
\]
and assume for the sake of contradiction that $\dim X<\beta$ and
$\Delta_{n}\geq2^{-kn}$ for some $k$. Since $\alpha\leq\dim X$,
we can choose $\varepsilon>0$ so that $\alpha<\beta-\varepsilon$.
Arguing analogously to Proposition \ref{prop:fiber-covering-numbers}
one obtains the analogous result:
\begin{prop}
\label{prop:entropy-large-fibers}There is a constant $c$ (depending
on $\beta,\varepsilon$) such that for large enough $m$,
\[
\mu^{(m)}\left(\bigcup\left\{ I\in\mathcal{I}_{r^{m}}\,:\, H_{r^{m+n}}(\mu_{I}^{(m)})>cm\right\} \right)>c
\]

\end{prop}
This lemma does not appear explicitly in \cite{Hochman2012a}, since
that is a direct proof. Ours is a proof by contradiction, and the
contradiction can be interpreted as showing that the lemma above is
false. This falsehood is demonstrated directly in the last displayed
equation of Section 5.3 of \cite{Hochman2012a}.

Next, for a probability measure $\nu$ and set $E$ with $\nu(E)>0$,
write $\nu_{E}$ for the conditional measure on $E$, that is, $\nu_{E}(A)=\frac{1}{\nu(E)}\nu(E\cap A)$.
The analog of Proposition \ref{prop:local-sumset} then holds, again
with an analogous proof:
\begin{prop}
[See Equation (40) of \cite{Hochman2012a}]\label{prop:local-convolution}For
every $\delta>0$, as $m\rightarrow\infty$ 
\[
\mu^{(m)}\left(\bigcup\left\{ I\in\mathcal{I}_{r^{m}}\,:\,\frac{1}{n}H_{r^{m+n}}(\mu_{I}^{(m)}*S_{r^{m}}\mu^{(n)})\leq\alpha+\delta\right\} \right)\geq1-o(1)
\]

\end{prop}
From the last two propositions one sees that for given $\delta>0$
and large enough $m$, there are intervals $I=I_{m}\in\mathcal{I}_{r^{m}}$
that appear in the unions in the conclusions of both propositions.
Taking $\nu_{m}$ to be the re-scaling of $\mu_{I}^{(m)}$ by $r^{-m}$
(translated back to $[0,1)$), we have the rigorous analog of ``Proposition''
\ref{prop:small-sumset}:
\begin{prop}
\label{prop:small-convolution}There is a $\tau>0$ such that for
every $\delta>0$, for all sufficiently large $m$, there is a measure
$\nu_{m}$ supported on $[0,1)$ with 
\begin{eqnarray}
\frac{1}{m}H_{r^{m}}(\nu_{m}) & > & \tau\label{eq:lerge-entropy}\\
\frac{1}{m}H_{r^{m}}(\mu^{(m)}*\nu_{m}) & < & \frac{1}{m}H_{r^{m}}(\mu^{(m)})+\delta\label{eq:small-convolution}
\end{eqnarray}

\end{prop}

\subsection{\label{sub:Getting-a-contradiction-really}Getting a contradiction}

The missing ingredient in Section \ref{sub:Getting-a-contradiction}
was the ability to estimate the number of leaves of a tree from the
average amount of branching of its nodes. This is where entropy really
comes in handy, because of the following (easy!) lemma. Recall that
given a tree-measure $\theta$, we write $\theta_{\sigma,m}$ for
the normalized weights on the nodes $m$ generations down from $\sigma$. 
\begin{lem}
[Lemma 3.4 of \cite{Hochman2012a}] \label{prop:entorpy-averages}Let
$\widetilde{\theta}$ be a tree-measure on the full binary tree $T$
of height $h$. Write $\partial\widetilde{\theta}$ for the measure
induced by $\widetilde{\theta}$ on the leaves of the tree. Then for
any $m$,
\[
\frac{1}{h}H(\partial\widetilde{\theta})=\mathbb{E}_{\sigma\sim\widetilde{\theta}}(\frac{1}{m}H(\widetilde{\theta}_{\sigma,m}))+O(\frac{m}{h})
\]

\end{lem}
From here the argument proceeds exactly as in Section \ref{sub:Getting-a-contradiction}.
Let $\tau>0$ be the constant provided by Proposition \ref{prop:small-convolution}.
Choose a small parameter $\varepsilon>0$. Choose $m$ large enough
that
\[
H_{r^{m}}(\mu^{(m)})\geq(1-\varepsilon)\alpha
\]
Apply the inverse theorem \ref{thm:inverse-thm-entropy} with parameters
$\varepsilon,m$ and let $\delta>0$ be the resulting number. Applying
Proposition \ref{prop:small-convolution} with this $\delta$, there
exist probability measures $\nu_{n}$ on $[0,1]$ satisfying (\ref{eq:lerge-entropy})
and (\ref{eq:small-convolution}). Write $\widetilde{\mu}^{(n)},\widetilde{\nu}_{n}$
for the lift of $\mu^{(n)},\nu_{n}$, respectively, to the binary
tree of height $h_{n}=\left\lceil 1/\log(r^{n})\right\rceil $. By
the inverse theorem there is a partition $U_{n}\cup V_{n}\cup W_{n}$
of the levels $\{1,\ldots,h_{n}\}$ such that
\begin{enumerate}
\item [(I)] At scales $i\in U_{n}$, the $\widetilde{\mu}^{(n)}$-mass
of nodes at level $i$ that are $\varepsilon$-uniform for $m$ generations
is at least $1-\varepsilon$.
\item [(II)] At scales $j\in V_{n}$, the $\widetilde{\nu}_{n}$-mass of
nodes at level $i$ that are $\varepsilon$-concentrated for $m$
generations is at least $1-\varepsilon$
\item [(III)] $|W_{n}|\leq\varepsilon h_{n}$.
\end{enumerate}
If $|V_{n}|>(1-\tau/2)h_{n}$ and $\varepsilon$ is small enough compared
to $\tau$, then sufficiently many nodes (with respect to $\widetilde{\nu}_{n}$)
would have $H(\widetilde{\nu}_{\sigma,m}^{n})<\varepsilon$ that we
could invoke Lemma \ref{prop:entorpy-averages} and conclude that
the entropy $H_{r^{n}}(\nu_{n})\approx\frac{1}{n\log(1/r)}H(\widetilde{\nu}^{n})<\tau$,
contradicting (\ref{eq:lerge-entropy}). Therefore $|V_{n}|\leq(1-\tau/2)h_{n}$.
In particular, assuming $\varepsilon$ is small enough compared to
$\tau$, 
\[
|U_{n}|\geq h_{n}-|V_{n}|-|W_{n}|\geq\frac{\tau}{3}h_{n}
\]
Next, suppose that $m$ is large enough so that $H_{r^{n}}(\mu^{(n)})>(1-\varepsilon)\alpha$.
Using self-similarity of $X$ and an argument analogous to the one
outlined in Proposition \ref{prop:SSSs-have-uniform-branching}, we
get the analogous result:
\begin{lem}
[Lemma 5.4 of \cite{Hochman2012a}]\label{lem:SSMs-are-entropy-uniform}For
all large enough $n$,
\[
\frac{1}{h_{n}+1}\sum_{\sigma}\left\{ \widetilde{\mu}^{(n)}(\sigma)\,:\,\frac{1}{m}H(\widetilde{\mu}_{\sigma,m}^{(n)})>(1-2\varepsilon)\alpha\right\} >1-\varepsilon
\]

\end{lem}
Now, from the definition of $U_{n}$ and our bound $|U_{n}|\geq\frac{\tau}{3}h_{n}$,
we know that at least a $(1-\varepsilon)\tau/3$-fraction of the nodes
of $\widetilde{\mu}^{(n)}$ satisfy $\frac{1}{m}H(\widetilde{\mu}_{\sigma,m}^{(n)})>(1-\varepsilon)m$.
Of the remaining nodes, by the last lemma all but a $\varepsilon$-fraction
satisfy $\frac{1}{m}H(\widetilde{\mu}_{\sigma,m}^{(n)})\geq(1-2\varepsilon)\alpha$.
Therefore by Lemma \ref{prop:entorpy-averages} again, for all large
enough $n$, 
\[
H_{r^{n}}(\mu^{(n)})\approx\frac{1}{h_{n}}H(\widetilde{\mu}^{(n)})>(1-\varepsilon)^{2}\frac{\tau}{3}+(1-(1-\varepsilon)\frac{\tau}{3}-\varepsilon)\alpha>\alpha+\varepsilon
\]
assuming $\varepsilon$ is small compared to $\tau$. This contradicts
the definition of $\alpha$.

\bibliographystyle{plain}
\bibliography{bib}

\end{document}